\documentclass{article}
\usepackage{graphicx, amsmath, amssymb, amsthm} 
\usepackage{xcolor,float}

\title{}
\author{}
\date{\today}

\newtheorem{thm}{Theorem}[section]
\newtheorem{lem}[thm]{Lemma}
\newtheorem{prop}[thm]{Proposition}
\newtheorem{cor}[thm]{Corollary}

\theoremstyle{definition}

\newtheorem{rem}[thm]{Remark}
\newtheorem{ex}[thm]{Example}

\usepackage[normalem]{ulem}
\usepackage{enumerate}
\usepackage[nospace,noadjust]{cite}
\usepackage{url}

\title{A note on the unknotting number and the region unknotting number of weaving knots}
\author{Ayaka Shimizu\thanks{Institute for Global Leadership, Ochanomizu University, Ohtsuka, Tokyo 1128610 Japan. Email: shimizu.ayaka@ocha.ac.jp, shimizu1984@gmail.com},
Amrendra Gill\thanks{Vellore Institute of Technology-Andhra Pradesh, Amaravati, Andhra Pradesh 522237, India.
Email: amrendrasgill@gmail.com},
and Sahil Joshi\thanks{Indian Institute of Science Education and Research Pune, Pune, Maharashtra 411008, India.
Email: sahil.joshi@acads.iiserpune.ac.in}}
\date{\today}
\begin{document}
\maketitle

\begin{abstract}
A weaving knot is an alternating knot whose minimal diagram is a closed braid of a lattice-like pattern. 
In this paper, the warping degree of a braid diagram is defined, and upper bounds of the unknotting number and the region unknotting number for some families of weaving knots are given by diagrammatical and combinatorial examination of the warping degree of weaving knot diagrams. 
\end{abstract}

\section{Introduction} %%%%%%%%%%%%%%%%%%

Let $B_W(p,q)$ be a braid of $p$ strands represented by
$$B_W(p,q) = \left( \sigma_1 \sigma_2^{-1} \sigma_3 \cdots \sigma_{p-1}^{(-1)^p}\right)^q$$ 
for positive integers $p$, $q$. 
A {\it weaving link} $W(p,q)$ is an alternating link of $\gcd(p,q)$ components which is defined as the closure of the braid $B_W(p,q)$, where $p \geq 3$, $q \geq 2$. 
In particular, when $\gcd(p,q)=1$, we call $W(p,q)$ a {\it weaving knot}. 
The torus link $T(p,q)$ and $W(p,q)$ share the same standard link projection (cf.~\cite{CKP}). 
Note that if $p = 2$, then the closure of $B_W(p,q)$ represents a torus link of the type $(2,q)$ for which most of the link invariants are explicitly known in the literature.
For $p = 1$ or $q = 1$, the closure of $B_W(p,q)$ always represents the trivial knot.
Throughout the paper, when we denote $B_W(p,q)$, we assume that $p$ and $q$ are positive integers. 
When we consider a weaving knot $W(p,q)$, assume that $p$ and $q$ are coprime integers with $p \geq 3$, $q \geq 2$.\par 
Invariants of weaving links have been a topic of interest in recent studies.
For instance, Champanerkar, Kofman and Purcell \cite{CKP} provide asymptotically sharp bounds of the hyperbolic volume of weaving links.
The signature of a weaving link is calculated by Mishra and Staffeldt in \cite{MS}, which coincides with Rasmussen's $s$-invariant and determines the support of the Khovanov homology.
Moreover, they also focus on the computations of polynomial invariants, the ranks of Khovanov homology groups, and higher twist numbers of the weaving link $W(p,q)$ for the case $p = 3$.
The link determinant of $W(p,q)$ for $p = 3,4$ or $q = 2$ is known (see for example \cite{KST,DMMS}).
Recently, explicit formulae of the Alexander and Jones polynomials of $W(p,q)$ for $p = 3$ are derived by AlSukaiti and Chbili in \cite{AC}.
It will be interesting to understand the complexity of a weaving knot or link by means of those invariants which are defined in a combinatorial manner using link diagrams.
Some of them are considered in the following discussions.

In this paper, we discuss upper bounds of the unknotting number $u(K)$ and the region unknotting number $u_R(K)$ of weaving knots $K$. 
The {\it unknotting number} $u(K)$ of a knot $K$ is the minimum number of crossing changes which are needed to transform $K$ into the trivial knot. 
We have the inequality $u(W(p,q)) \leq (p-1)q/2 -1$ (Corollary \ref{cor-uc-1} in Section \ref{section-kt}) from a relation between the unknotting number and crossing number. 
Moreover, for some specific types of weaving knots, we have the following inequalities. 

\begin{thm}
Let $p$ be an odd integer with $p \geq 3$, $n$ be a non-negative integer and $r$ be an integer with $1 \leq r \leq p-1$ and $\gcd(p,r)=1$.
Then 
$$u(W(p,np+r)) \leq \frac{n(p^2-1)}{4}+\frac{(p-1)r}{2}-1$$
holds.
\label{thm-42-1}
\end{thm}

\begin{thm}
When $p$ is an odd integer with $p \geq 3$, 
$$u(W(p,np+1)) \leq \frac{n(p^2-1)}{4}$$
holds for any positive integer $n$ and
$$u(W(p,np+2)) \leq \frac{n(p^2-1)}{4} + \frac{p-1}{2}$$
holds for any non-negative integer $n$. 
When $p$ is an even integer with $p \geq 4$, 
$$u(W(p,np+1)) \leq \frac{np(p-1)}{2}$$
holds for any positive integer $n$.
\label{thm-42-2}
\end{thm}

\noindent The upper bounds given in these theorems are attained for certain weaving knots mentioned in the following remark.
For their Rolfsen's name, one may refer to Livingston and Moore~\cite{LM}, a database of knot invariants, or \cite{MS}. 
\begin{rem}
For the inequalities given in Theorems~\ref{thm-42-1} and \ref{thm-42-2}, there exist the following examples for which the equality holds.
\begin{enumerate}
\item If $p=3$ and $n=r=1$, then $W(p,np+r)=W(3,4) = 8_{18}$ for which $u(W(3,4))=2 \leq \frac{1(9-1)}{4} + \frac{(3-1)}{2} - 1 = 2$.
However, the upper bound is not attained if $p=3$, $n=1$, and $r=2$ for which $W(p,np+r)=W(3,5) = 10_{123}$, and $u(W(3,5)) = 2 \leq \frac{1(9-1)}{4} + \frac{(3-1)2}{2} - 1 = 3$.
\item If $p = 5$ and $n = 0$, then $W(p,np+2)=W(5,2) = 8_{12}$ for which $u(W(5,2)) = 2 \leq 0 + \frac{(5-1)}{2} = 2$.
\end{enumerate}
\end{rem}

\noindent For a knot diagram $D$ and a region $R$ of $D$, a {\it region crossing change on $R$} is a set of crossing changes at all the crossings on the boundary of $R$ (see \cite{ASr}). 
The {\it region unknotting number of a knot diagram $D$}, $u_R(D)$, is the minimum number of region crossing changes which are required to transform $D$ into a diagram of the trivial knot. 
The {\it region unknotting number of a knot $K$}, $u_R(K)$, is the minimum value of $u_R(D)$ for all minimal crossing diagrams $D$ of $K$. 
For weaving knots $W(p,q)$, we have the inequality $u_R(W(p,q)) \leq (p-1)q/2 + 1/2$ (Corollary \ref{cor-ur-c12} in Section \ref{section-rcc}) from a relation between the region unknotting number and the crossing number. 
For the following types of weaving knots, we have sharp upper bounds. 

\begin{thm}
When $p$ is an odd integer with $p \geq 3$, 
$$u_R(W(p,np+1)) \leq \frac{n(p^2-1)}{4}$$
holds for any positive integer $n$ and 
$$u_R(W(p,np+2)) \leq \frac{n(p^2-1)}{4}+\frac{p-1}{2}$$
holds for any non-negative integer $n$. 
\label{thm-r12}
\end{thm}

\begin{rem}
For the inequalities given in Theorem~\ref{thm-r12}, there exist the following weaving knots (see~\cite{ASr}) for which the equality holds.
\begin{enumerate}
\item If $p=3$ and $n=1$, then $W(p,np+1)=W(3,4) = 8_{18}$ for which $u_R(W(3,4))=2 \leq \frac{1(9-1)}{4} = 2$.
\item If $p=3$ and $n=0$, then $W(p,np+2)=W(3,2) = 4_1$ for which $u_R(W(3,2)) = 1 \leq 0 + \frac{(3-1)}{2} = 1$.
\end{enumerate}
\end{rem}

\noindent In the proofs of Theorems \ref{thm-42-1}, \ref{thm-42-2} and \ref{thm-r12}, we use the ``warping degree'' which represents a complexity of a knot diagram, effectively taking advantage of the regular lattice-like pattern of weaving knots. 
The rest of the paper is organized as follows: 
In Section \ref{section-preliminaries}, we see some basic and useful properties of weaving knots. 
In Section \ref{section-wd}, we study the warping degree. 
We also define the warping degree for braid diagrams. 
In Section \ref{section-pf}, Theorems \ref{thm-42-1} and \ref{thm-42-2} are shown. 
In Section \ref{section-rcc}, we discuss the region unknotting numbers and prove Theorem \ref{thm-r12}. 
In Appendix, we investigate the warping degree of weaving knot diagrams by considering the isolate-region numbers.

\section{Preliminaries on weaving knots} %%%%%%%%%%%%%%%%%%
\label{section-preliminaries}

In this section, we see some properties of weaving knots which are used in this paper. 

\subsection{Relation between two strands in $B_W(p,p)$} %%%%%%%%%%%%%%%%%%

\noindent For the canonical braid diagram of $B_W(p,p)$, we have the following proposition. 

\begin{prop}
On the canonical diagram of $B_W(p,p)$, each pair of strands have exactly two mutual crossings. 
Moreover, on the two crossings, when $p$ is an odd number, one of the two strands is over at both crossings. 
When $p$ is an even number, each of the two strands has one over-crossing and one under-crossing between them.
\label{pp-c}
\end{prop}

\begin{proof}
For the braid $B_W(p,p)= \left( B_W(p,1) \right)^p$, we count each $B_W(p,1)$ as ``one round''. 
We call the $i$th strand from the left-hand side at the top of the braid the ``$i$th strand''. 
Let $1 \leq i < j \leq p$, and let $s_i$, $s_j$ be the $i$th, $j$th strands of $B_W(p,p)$, respectively. 
The strands $s_i$ and $s_j$ meet at the ($j-i$)th crossing in the $i$th round once, and the ($i+p-j$)th crossing in the $j$th round. 
Note that other than the $i$ and $j$th rounds, they have no crossings. 
When $p$ is odd, the signs of $\sigma_{j-i}$ in the $i$th round and $\sigma_{p-(j-i)}$ in the $j$th round in $B_W(p,p)$ are opposite because they have different parities of the subscripts. 
This implies that at one crossing, the arc heading in the right direction is over, and at the other crossing the arc heading in the left is over. 
Hence one of $s_i$ and $s_j$ is over at both crossings. 
We note that in each $k$th round, only $k$th strand is heading in the right direction. 
When $p$ is even, the signs of $\sigma_{j-i}$ and $\sigma_{p-(j-i)}$ are same, and one strand is over at one crossing and the other strand is over at the other crossing. 
\end{proof}

\subsection{Minimal diagram} %%%%%%%%%%%%%%%%%%

In this subsection, we show that each weaving knot has a unique minimal diagram on $S^2$. 
The following theorem was proved by Menasco in \cite{WM}. 

\begin{thm}[\cite{WM}]\label{wm}
Let $D$ be a reduced alternating diagram of an alternating link $L$.
\begin{description}
\item[(a)] If $D$ is connected, then $L$ is non-split.
\item[(b)] If $L$ is non-split, then $L$ is prime if for each circle $C$ on $S^2$ that intersects exactly two points on edges of $D$ transversely has no crossings in one side of $C$.
\end{description}  
\end{thm}

\noindent For weaving links, we have the following. 

\begin{prop}
Every weaving link $W(p,q)$ is non-split and prime. 
\end{prop}

\begin{proof}
Let $D$ be the closure of the braid diagram $B_W(p,q)$.
\begin{description}
\item[(a)] Since the reduced alternating diagram $D$ is connected, $W(p,q)$ is a non-split link by Theorem~\ref{wm} (a).
\item[(b)] The diagram $D$ has two $q$-gons and $2q$ of 3-gons, and the others are 4-gons. 
The adjacent regions of each $q$-gon are $q$ distinct 3-gons. 
The adjacent regions of each 3-gon are two distinct 4-gons (3-gons when $p=3$) and a $q$-gon. 
The adjacent regions of each 4-gon are 4 distinct 3-gons or 4-gons. 
Hence, any pair of regions does not share two distinct edges. 
Therefore, any circle $C$ which intersects two points on $D$ has the intersection points on the same edge of $D$, and in one side of $C$, there are no crossings.
Thus, by applying Theorem~\ref{wm} (b), we have that $W(p,q)$ is prime. 
\end{description}
\end{proof}

\noindent It is known that any reduced alternating diagram of an alternating knot is a minimal diagram and any non-alternating diagram of a prime alternating knot cannot be minimal (the Tait's first conjecture, proved independently by Kauffman~\cite{LK}, Murasugi~\cite{KM}, and  Thistlethwaite~\cite{Th}). 
It is also known that any pair of reduced alternating diagram of an alternating knot are related by a finite sequence of flypings (the Tait flyping conjecture, proved in \cite{MTa, MTc} by Menasco and Thistlethwaite), where a {\it flyping} is a transformation on knot diagrams shown in Figure~\ref{fig-flyping}. 
We say a flyping on a knot diagram is {\it non-trivial} if the resulting knot diagram is different from the original one. 
The contraposition of the following proposition is useful. 
\begin{figure}[ht]
\centering
\includegraphics[width=6cm]{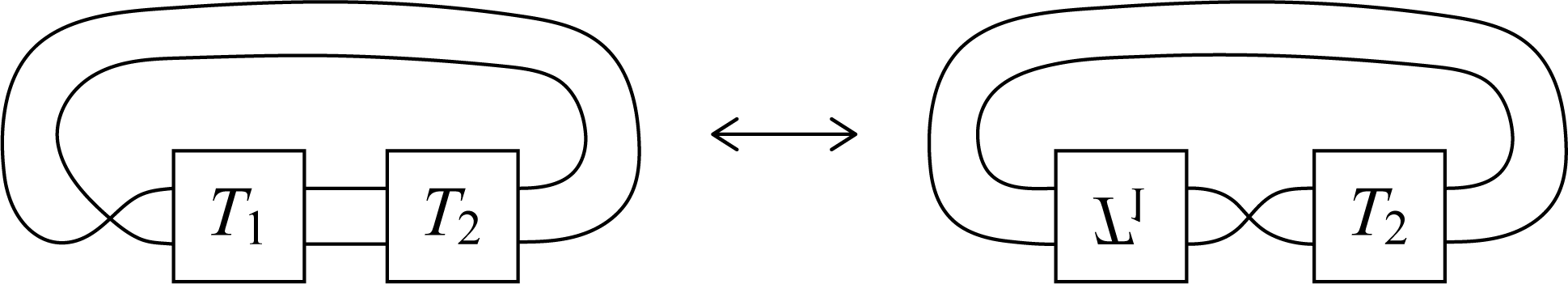}
\caption{A flype on an alternating diagram. Assume that the two crossings have visually same over/under information.}
\label{fig-flyping}
\end{figure}

\begin{prop}
If a link diagram $D$ admits a non-trivial flyping, there exists a circle $C$ which intersects a crossing exactly once and edges exactly twice transversely, and has crossings in both sides. 
\label{prop-flype}
\end{prop}

\begin{proof}
See Figure \ref{fig-flyping2}. 
\begin{figure}[ht]
\centering
\includegraphics[width=2.5cm]{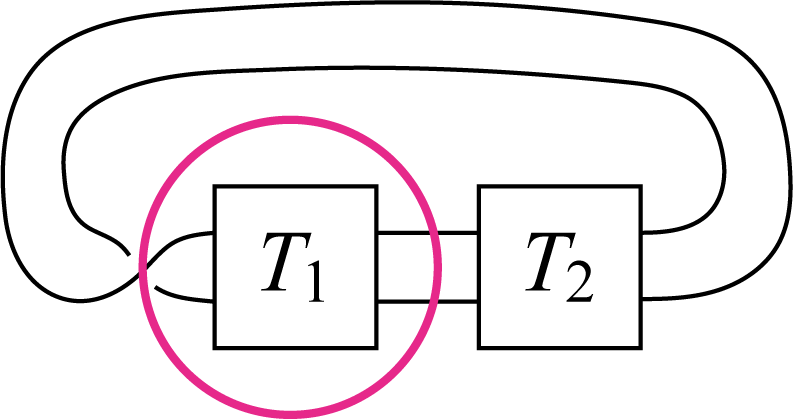}
\caption{Flyping can be applied at the circle. }
\label{fig-flyping2}
\end{figure}
\end{proof}

\noindent We remark that the converse of Proposition \ref{prop-flype} is not true. 
See, for example, Figure \ref{fig-flyping3}\footnote{
It is known that every alternating knot with Conway's notation $mn$ has only one minimal diagram (see, for example, \cite{calvo, ASr}).}. 
For weaving knots, we have the following. 
\begin{figure}[ht]
\centering
\includegraphics[width=2.5cm]{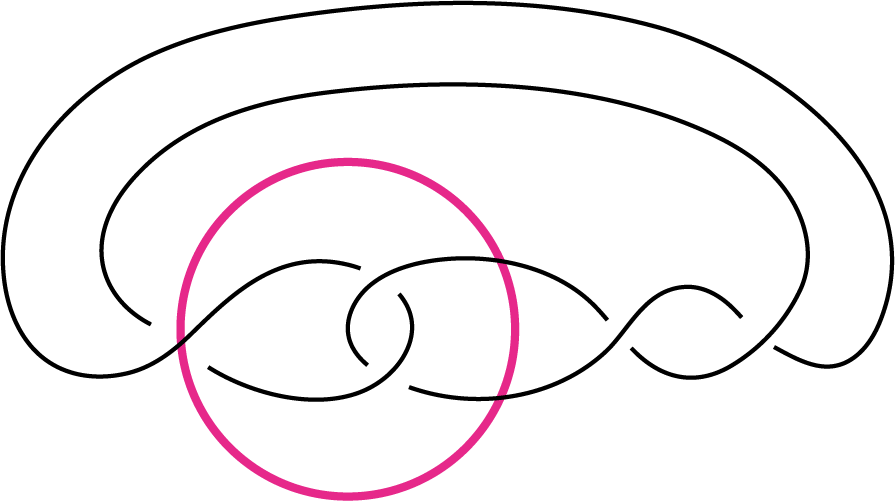}
\caption{Trivial flyping at the circle.}
\label{fig-flyping3}
\end{figure}

\begin{prop}
Each weaving knot $W(p,q)$ has only one minimal diagram on $S^2$. 
\label{umd}
\end{prop}

\begin{proof}
Let $D$ be the closure of the braid diagram $B_W(p,q)$, which is a reduced alternating diagram, namely, a minimal diagram of a weaving knot $W(p,q)$. 
On $D$, for any pair of regions of $D$ which share a crossing $c$ diagonally, only the other two regions sharing $c$ are the common adjacent regions. 
Then, any circle which intersects $D$ at a crossing $c$ only once and edges only twice has no crossings in one side. 
Therefore, by Proposition \ref{prop-flype} we cannot apply a non-trivial flyping on $D$ to obtain another minimal diagram. 
\end{proof}

\subsection{Unknotting number and crossing number} %%%%%%%%%%%%%%%%%%%%%%%%%%%%%
\label{section-kt}

In this subsection, we see the upper bounds for the unknotting number of weaving knots and links which immediately follow Taniyama's theorems. 
Let $c(D)$ denote the crossing number, the number of crossings, of a link diagram $D$.
Let $u(D)$ denote the minimum number of crossing changes which are needed to transform the link diagram $D$ into a diagram of the trivial knot or link. 
For any knot diagram $D$, the inequality $u(D) \leq \frac{c(D)-1}{2}$ is well known. 
Let $c(K)$ be the crossing number of $K$, the minimum value of $c(D)$ over all diagrams $D$ of $K$. 
For a non-trivial knot $K$, the inequality $u(K) \leq \frac{c(K)-1}{2}$ follows immediately.  
Moreover, the following theorem was shown by Taniyama in \cite{kt}. 

\begin{thm}[\cite{kt}] 
We have the following. 
\begin{itemize} 
\item[(1)] Let $D$ be a knot diagram that satisfies $u(D)=\frac{c(D)-1}{2}$. Then, $D$ is a reduced alternating diagram of some $(2,q)$-torus knot, or $D$ is a diagram with just one crossing. 
\item[(2)] Let $K$ be a knot that satisfies $u(K)= \frac{c(K)-1}{2}$.
Then, $K$ is a $(2,q)$-torus knot for some odd number $q \neq \pm 1$.
Namely, only 2-braid knots satisfy the equality. 
\end{itemize}
\label{kt-k}
\end{thm}

\noindent It is known that weaving knots $W(p,q)$ are not $(2,q)$-torus knots when $p\geq 3$ because the former is a hyperbolic knot whereas the latter is not for any $q$ (see~\cite{CKP}).
By Proposition \ref{umd}, we have $c(W(p,q))=(p-1)q$.
Thus, applying the contraposition of Theorem~\ref{kt-k}, we have the following corollary.

\begin{cor}
For weaving knots $W(p,q)$, we have 
$$u(W(p,q)) \leq \frac{c(W(p,q))-2}{2} = \frac{(p-1)q}{2}-1.$$
\label{cor-uc-1}
\end{cor}

\noindent For links $L$, it is well known that the inequality $u(D) \leq \frac{c(D)}{2}$ holds for any diagram $D$ of $L$, and consequently $u(L) \leq \frac{c(L)}{2}$ also holds. 
In \cite{kt}, the following theorem was also shown. 

\begin{thm}[\cite{kt}]
\begin{itemize}
\item[(1)] Let $D= K_1 \cup K_2 \cup \dots \cup K_r$ be a diagram of an $r$-component link that satisfies $u(D)= \frac{c(D)}{2}$. Then, each $K_i$ is a simple closed curve on $S^2$ and for each pair $i, j$, the subdiagram $K_i \cup K_j$ is an alternating diagram or a diagram without crossings. 
\item[(2)] Let $L$ be an $r$-component link that satisfies $u(L)= \frac{c(L)}{2}$. Then $L$ has a diagram $D = K_1 \cup K_2 \cup \dots \cup K_r$ such that each $K_i$ is a simple closed curve on $S^2$ and for each pair $i, j$, the subdiagram $K_i \cup K_j$ is an alternating diagram or a diagram without crossings. 
\end{itemize}
\label{kt-l}
\end{thm}

\noindent We have the following corollary. 

\begin{cor}
For a weaving link $W(p,np)$, if $p$ is a positive odd integer and $n$ is a positive integer, we have
$$u(W(p,np)) \leq \frac{c(W(p,np))-1}{2} = \frac{(p-1)np-1}{2}.$$
\end{cor}

\begin{proof}
Let $D$ be the closure of the braid diagram $B_W(p,np)$. 
We note that $W(p,np)$ is a link of $p$ components and the $p$ strands of $B_W(p,np)$ belong to different components of $W(p,np)$. 
Since $p \leq np$, $B_W(p,p)$ is a subbraid of $B_W(p,np)$.
By Proposition \ref{pp-c}, for any odd number $p$, subdiagram of $B_W(p,p)$ consisting of any two distinct strands $s_i$ and $s_j$ is non-alternating.
Therefore, the subdiagram of $D$ consisting of any two components is also non-alternating, and hence, by the contraposition of Theorem \ref{kt-l} (1), we have $u(D)\neq \frac{c(D)}{2}$. 
\end{proof}

\section{Warping degree} %%%%%%%%%%%%%%%%%%%%
\label{section-wd}

In this section, we discuss the warping degree of braid diagrams and their closures to estimate the unknotting number of weaving knots. 

\subsection{Warping degree of a link diagram} %%%%%%%%%%%%%%%%%%%%

Let $D$ be an oriented $n$-component link diagram. 
Take a base point $b_i$ on each component of $D$. 
We denote the pair of a diagram $D$ and a sequence of base points $\mathbf{b} =\{ b_1, b_2, \dots , b_n \}$ by $D_{\mathbf{b}}$. 
A crossing $c$ is said to be a {\it warping crossing point of $D_{\mathbf{b}}$} if one encounters $c$ as an under-crossing first when one travels $D$ from $b_1$ to $b_1$, $b_2$ to $b_2$, $\dots$, $b_n$ to $b_n$ according to the orientation (\cite{K-l, ASw}). 
Let $D_i$ denote the knot component of $D_{\mathbf{b}}$ which has the base point $b_i$. 
We note that when $c$ is a crossing between $D_i$ and $D_j$ ($i<j$), $c$ is a warping crossing point of $D_{\mathbf{b}}$ for $\mathbf{b}=(b_1, b_2, \dots , b_n)$ if and only if $D_i$ is under $D_j$ at $c$. 

The {\it warping degree of $D_{\mathbf{b}}$}, $d(D_{\mathbf{b}})$, is the number of the warping crossing points of $D_{\mathbf{b}}$. 
The {\it warping degree of $D$}, $d(D)$, is the minimum value of $d(D_{\mathbf{b}})$ for all sequences of base points $\mathbf{b}$. 
It is known that a link diagram represents a trivial link if $d(D)=0$ (see, for example, \cite{K-l, KKT}). 

\subsection{Warping degree of a braid diagram} %%%%%%%%%%%%%%%%%%%%

In this subsection, the canonical diagram of the braid $B_W(p,q)$ is also denoted by $B_W(p,q)$. 
We define the warping degree for braid diagrams\footnote{Recently, in \cite{NSP}, the ``warping labeling'' for a braid diagram and its generalizations are defined following properties of the warping degrees of the closure. The ``warping degree'' of a braid diagram defined in this paper is different from them; it takes a numerical value and is computable without considering the closure.}. 
Let $B$ be a braid diagram of $n$ strands. 
Take a base point $b_i$ at the top of each strand. 
We denote the pair of $B$ and a sequence of base points $\mathbf{b}$ by $B_{\mathbf{b}}$. 
Let $s_i$ denote the strand of $B$ which has the base point $b_i$. 
For a crossing $c$ between strands $s_i$ and $s_j$ where $b_i$ is positioned ahead of $b_j$ in the sequence $\mathbf{b}$, $c$ is said to be a {\it warping crossing point of $B_{\mathbf{b}}$} if $s_i$ is under $s_j$ at $c$. 
The {\it warping degree of $B_{\mathbf{b}}$}, denoted by $d(B_{\mathbf{b}})$, is the number of the warping crossing points of $B_{\mathbf{b}}$. 
The {\it warping degree of $B$}, $d(B)$, is the minimum value of $d(B_{\mathbf{b}})$ for all sequences $\mathbf{b}$ of base points $b_i$ at the top of the strands. 

\begin{ex}
Let $B=B_W(7,7)$ with base points $b_1, b_2, \dots , b_7$ as shown in Figure \ref{fig-w77}. 
If we take a sequence of base points $\mathbf{b}=(b_1, b_2, b_3, b_4, b_5, b_6, b_7)$, we have $d(B_{\mathbf{b}})=22$. 
If we take $\mathbf{b}'=(b_1, b_3, b_5, b_7, b_2, b_4, b_6)$, we have $d(B_{\mathbf{b}'})=12$. 
\begin{figure}[h]
\centering
\includegraphics[width=4cm]{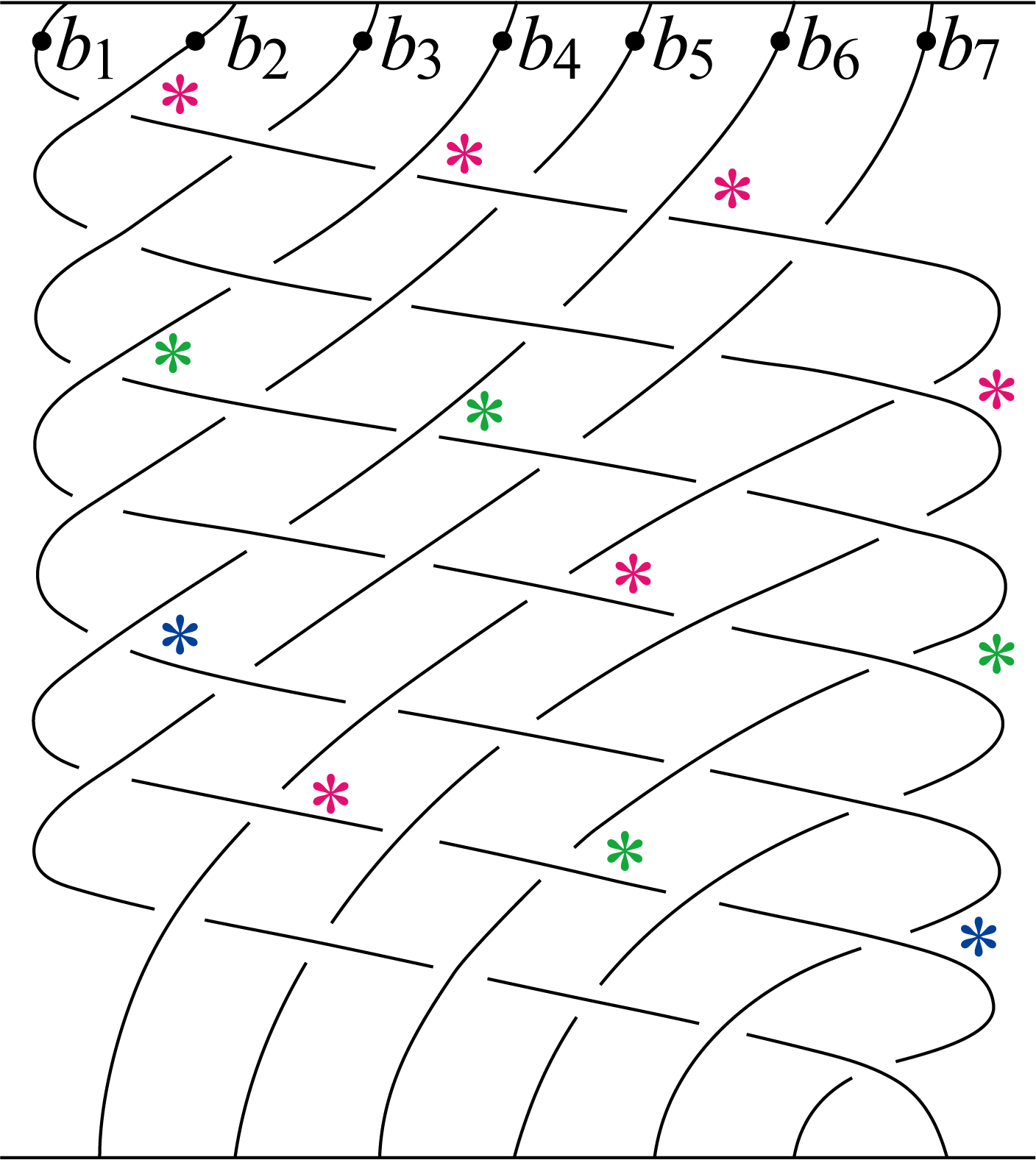}
\caption{The braid diagram $B_W(7,7)$ has warping degree 12 with the sequence of base points $\mathbf{b}'=(b_1, b_3, b_5, b_7, b_2, b_4, b_6)$.}
\label{fig-w77}
\end{figure}
\end{ex}

\noindent Let $B$ be a braid diagram. 
When a strand of $B$ is positioned at the $i$th on the top and at the $j$th on the bottom, we denote $\rho (i)=j$. 
For example, the braid diagram in Figure \ref{fig-b5} has $\rho(1)=3$, $\rho(2)=5$, $\rho(3)=1$, $\rho(4)=2$, $\rho(5)=4$. 
We also denote it by $\rho(1,2,3,4,5)=(3,5,1,2,4)$. 
For pure braids, we have the following lemma. 
\begin{figure}[h]
\centering
\includegraphics[width=2.3cm]{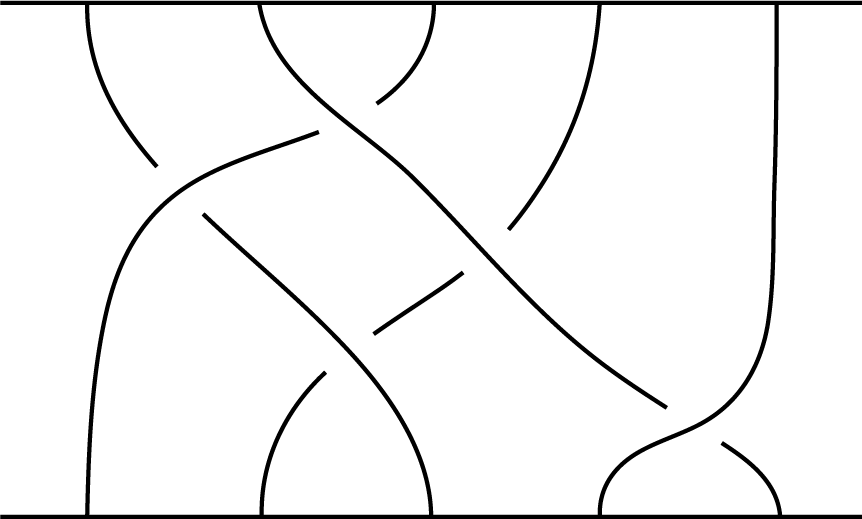}
\caption{$\rho (1,2,3,4,5)=(3,5,1,2,4)$.}
\label{fig-b5}
\end{figure}

\begin{lem}
Let $B$ be a pure braid diagram, namely, a braid diagram with the identity permutation $\rho (1,2, \dots ,n)=(1,2, \dots ,n)$. 
If $d(B)=0$, then $B$ is equivalent to the trivial braid. 
\label{d0-t}
\end{lem}

\begin{proof}
Let $B$ be a pure braid diagram of $d(B)=0$ with $n$ strands. 
Let $\mathbf{b}$ be a sequence of base points of $B$ such that $d(B_{\mathbf{b}})=d(B)=0$. 
Let $D$ be the closure of $B$. 
Then $D$ is a diagram of an $n$-component trivial link because $d(D_{\mathbf{b}})=0$ with the same $\mathbf{b}$. 
Fixing the closure part, $B$ can be transformed into the braid diagram with no crossings. 
\end{proof}

\noindent We remark that Lemma \ref{d0-t} does not hold when the permutation of a braid is not identity. 
For example, the braid diagram with $\rho (1,2,3)=(3,2,1)$ depicted in Figure \ref{fig-c-ex} has warping degree zero and the closure is the Hopf link, which is a nontrivial link. 
We will discuss more in Subsection \ref{section-wd-closure}. 
\begin{figure}[h]
\centering
\includegraphics[width=1.7cm]{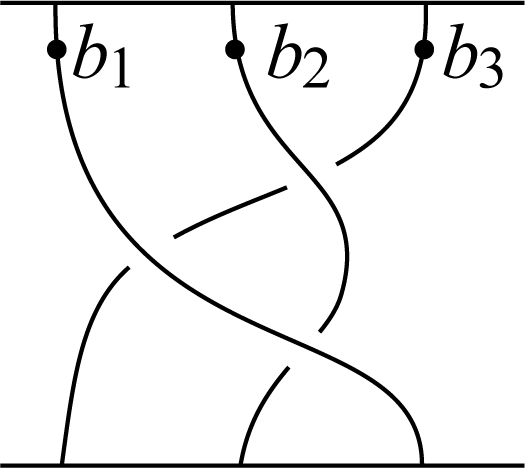}
\caption{A braid diagram with warping degree zero.}
\label{fig-c-ex}
\end{figure}

\noindent By taking suitable sequence of base points, we have the following lemmas for $B_W(p,p)$. 

\begin{lem}
When $p$ is a positive odd number, 
$$d(B_W(p,p)) \leq \frac{p^2-1}{4}.$$
\label{pp-o}
\end{lem}

\begin{proof}
Take base points $b_1, b_2, \dots , b_p$ at the top of strands from left to right. 
Let $s_i$ denote the strand which has the base point $b_i$. 
By Proposition \ref{pp-c}, each pair of strands have exactly two crossings, and at both crossings one of the two strands is over than the other one.  
More precisely, when $i$ is an odd number, $s_i$ is under $s_j$ for all even numbers $j$ with $i<j$ and all odd numbers $j$ with $i>j$. 
When $i$ is an even number, $s_i$ is under $s_j$ for all odd numbers $j$ with $i<j$ and all even numbers $j$ with $i>j$. 
Take a sequence of base points $\mathbf{b}=(b_1, b_3, b_5, \dots , b_p, b_2, b_4, b_6, \dots , b_{p-1})$. 
Then for each odd number $i$, the strand $s_i$ has warping crossing points at crossing with $s_j$ for all even number $j$ with $i<j$, and the number of warping crossing points on $s_i$ is $\left(\frac{p-1}{2}-\frac{i-1}{2}\right)\times 2 = p-i$. 
Note that there are no warping crossing points between $s_i$ and $s_j$ when both $i$ and $j$ are even or odd. 
Hence, 
\begin{align*}
d(B_W(p,p))=\sum_{i \text{ odd}}(p-i)
= \sum_{k=1}^{\frac{p+1}{2}} \left( p-(2k-1) \right)
= \frac{p^2-1}{4}.
\end{align*} 
\end{proof}

\begin{lem}
When $p$ is a positive even number, 
$$d(B_W(p,p))=\frac{p(p-1)}{2}.$$
\label{pp-e}
\end{lem}

\begin{proof}
Any pair of strands have exactly two crossings, where one strand is over at one crossing and the other one is over at the other crossing by Proposition \ref{pp-c}. 
This means that exactly one of the two crossings is a warping crossing point for any order of base points. 
Hence $d(B_W(p,p)_{\mathbf{b}})$ is equal to half the total number of crossings for any sequence $\mathbf{b}$ of base points. 
\end{proof}

\subsection{Warping degree and unknotting number}%%%%%%%
\label{section-wd-closure}

In this subsection, we consider the warping degree of the closure of a braid diagram and discuss the unknotting number. 
Let $B$ be a braid diagram of $n$ strands with braid permutation $\rho$. 
Let $\mathbf{b}=(b_1, b_2, \dots , b_n)$ be a sequence of base points of $B$, where $b_1, b_2, \dots , b_n$ are positioned from left to right. 
Let $\mathbf{a}$ be a sequence of the base points $b_1, b_2, \dots , b_n$ with some order. 
We say $\mathbf{a}$ {\it follows the closure} when the next component to $b_i$ is $b_{\rho (i)}$ if there does not exist $b_{\rho (i)}$ before $b_i$. 
For example, $\mathbf{b}_1=(b_1, b_3, b_4, b_2, b_5)$ and $\mathbf{b}_2=(b_5, b_4, b_2, b_3, b_1)$ follow the closure, whereas $\mathbf{b}_3=(b_5, b_2, b_3, b_1, b_4)$ does not follow the closure for the braid diagram in Figure \ref{fig-b5} with the base points shown in Figure \ref{fig-b5-c}. 
For a braid diagram $B$, we define the {\it closed warping degree}, $\overline{d}(B)$, to be the minimum value of $d(B_{\mathbf{b}})$ for all sequences of base points $\mathbf{b}$ which follow the closure. 
For example, for the braid diagram $B$ in Figure \ref{fig-b5-c}, we have $\overline{d}(B) \leq d(B_{\mathbf{b}_2})=2$. 
We note that $\overline{d}(B) \neq d(B) = d(B_{\mathbf{b}_3})=0$ (we can check $\overline{d}(B) \neq 0$ by Proposition \ref{dd-inequality} and the linking number of the closure.)
We have the following inequality.

\begin{prop}
Let $D$ be the closure of a braid diagram $B$ with the same orientation to $B$. 
Then the inequality $d(D) \leq \overline{d}(B)$ holds.
\label{dd-inequality}
\end{prop}

\begin{proof}
Let $\mathbf{b}$ be a sequence of base points which follows the closure with $d(B_{\mathbf{b}}) = \overline{d}(B)$. 
If the closure $D$ of $B$ represents an $r$-component link, select the base point of each component of the link which appears in the sequence $\mathbf{b}$ for the first time. 
Then we obtain a sequence $\mathbf{c}$ of $r$ base points for $D$, where the components of $\mathbf{c}$ follows the order of $\mathbf{b}$. 
Since we can travel the diagrams $B$ and $D$ in the same order according to $\mathbf{b}$ and $\mathbf{c}$, $B_{\mathbf{b}}$ and $D_{\mathbf{c}}$ have warping crossing points at the same crossings. 
Hence, $d(D) \leq d(D_{\mathbf{c}})= d(B_{\mathbf{b}}) = \overline{d}(B)$. 
\end{proof}

\begin{figure}[h]
\centering
\includegraphics[width=3.5cm]{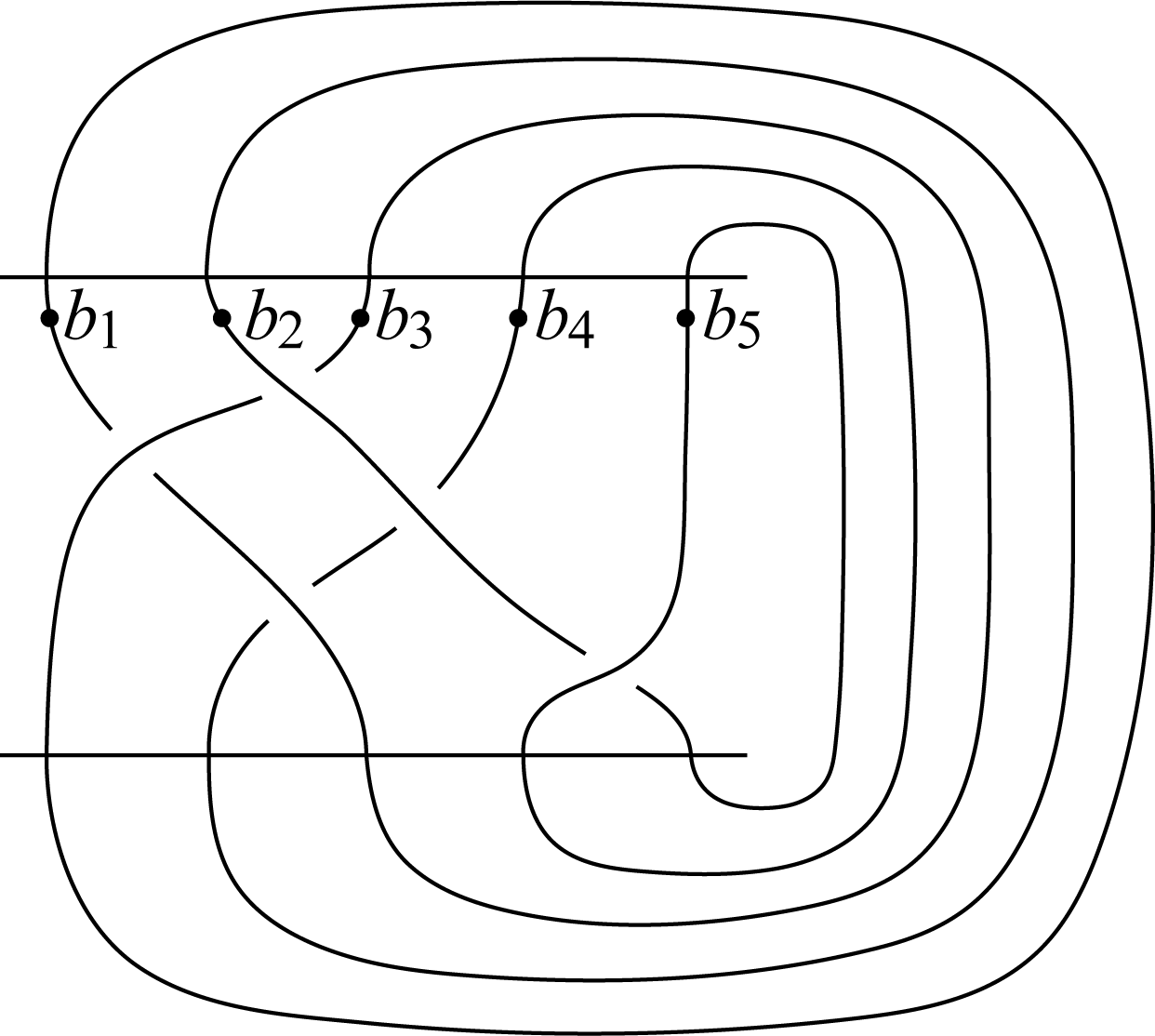}
\caption{$d(B_{(b_5, b_4, b_2, b_3, b_1)})=d(D_{(b_5, b_3)})$.}
\label{fig-b5-c}
\end{figure}

\noindent For example, for the braid diagram $B$ and its closure $D$ in Figure \ref{fig-b5-c}, we have $d(D) \leq \overline{d}(B) \leq 2$. 
We have the following corollary. 

\begin{cor}
If a braid diagram $B$ satisfies $\overline{d}(B)=0$, then the closure $D$ of $B$ represents a trivial knot or link. 
\label{cor-d0}
\end{cor}

Let $c(B)$ denote the number of crossings of a braid diagram $B$. 
Let $u(B)$ denote the minimum number of crossing changes which are required to transform $B$ into a braid whose closure is a trivial knot or link. 
We have the following. 

\begin{prop}
We have 
$$u(B) \leq \frac{c(B)}{2}.$$
\end{prop}

\begin{proof}
Let $D$ be the closure of a braid diagram $B$. 
Since there are no crossings in the closure part, we have $u(B)=u(D)$ and $c(B)=c(D)$. 
Then the inequality follows from the inequality $u(D) \leq \frac{c(D)}{2}$.  
\end{proof}

\noindent Since $c(B_W(p,q))=(p-1)q$, we have $u(B_W(p,q)) \leq (p-1)q/2$. 
Moreover, by the contraposition of Theorem \ref{kt-k}, we have the following corollary.

\begin{cor}
For any integers $p \geq 3$ and $q$ with $\gcd(p,q)=1$, we have
$$u(B_W(p,q)) \leq \frac{(p-1)q}{2}-1.$$ 
\label{cor-ub}
\end{cor}

\noindent Note that in Corollary~\ref{cor-ub}, the braid $B_W(p,q)$ is trivialized by crossing changes, whereas in Corollary~\ref{cor-uc-1}, its closure $W(p,q)$ is transformed into the trivial knot.
Regarding $u(B_W(p,q))$ for $q=1,2$, we have the following lemmas. 

\begin{lem}
We have $u(B_W(p,1))=0$ for any positive integer $p$. 
\label{u-r1}
\end{lem}

\begin{proof}
Since $B_W(p,1)$ has only one round, we can apply the Reidemeister moves of type I on the closure iteratively until all the crossings are reduced. 
\end{proof}

\begin{lem}
We have 
$$u(B_W(p,2)) \leq \frac{p-1}{2}$$
for any odd number $p \geq 3$. 
\label{u-r2}
\end{lem}

\begin{proof}
Apply a crossing change at the $2, 4, \dots , (p-1)$th crossings in the first round. 
Then the closure of $( \sigma_1 \sigma_2 \cdots \sigma_{p-2}\sigma_{p-1})( \sigma_1 \sigma_2^{-1} \cdots \sigma_{p-3}^{-1} \sigma_{p-2}\sigma_{p-1}^{-1})$ can be applied Reidemeister moves of type I\hspace{-0.5pt}I\hspace{-0.5pt}I, I, and I\hspace{-0.5pt}I on the right-hand side, as shown in Figure \ref{fig-r3-r1}, to obtain the closure of $( \sigma_1 \sigma_2 \cdots \sigma_{p-3} \sigma_{p-2})( \sigma_1 \sigma_2^{-1} \cdots \sigma_{p-3}^{-1} )$. 
Next, apply a Reidemeister move of type I. 
Then the braid representation will be $( \sigma_1 \sigma_2 \cdots \sigma_{p-3})( \sigma_1 \sigma_2^{-1} \cdots \sigma_{p-3}^{-1} )$. 
Repeat the set of Reidemeister moves iteratively until all the crossings are reduced. 
\begin{figure}[h]
\centering
\includegraphics[width=9cm]{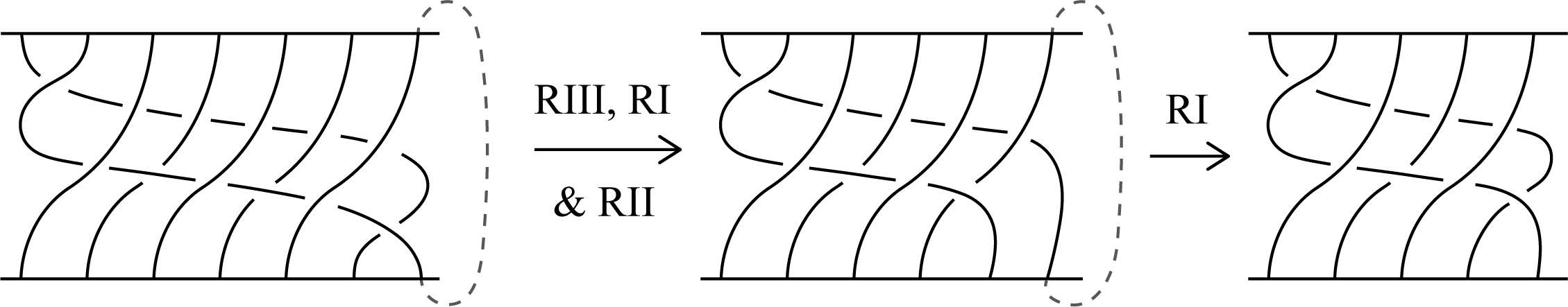}
\caption{Transformations.}
\label{fig-r3-r1}
\end{figure} 
\end{proof}

\section{Proofs of Theorems \ref{thm-42-1}, \ref{thm-42-2}} %%%%%%%%%%%%%%%%%%%%%%%%%%%%%
\label{section-pf}

In this section, proofs of Theorems \ref{thm-42-1} and \ref{thm-42-2} are given. 

\medskip 
\noindent {\it Proof of Theorem \ref{thm-42-1}.} 
The braid diagram $B_W(p,np+r)$ can be decomposed into $n$ braids of $B_W(p,p)$ and a single $B_W(p,r)$. 
After the crossing changes at the warping crossing points shown in the proof of Lemma \ref{pp-o}, each $B_W(p,p)$ part will be equivalent to the trivial braid by Lemma \ref{d0-t}. 
For the remaining part $B_W(p,r)$, we have the inequality of Corollary \ref{cor-ub}. 
Hence, $u(W(p,np+r)) \leq u(B_W(p,np+r)) \leq n \times d(B_W(p,p)) +u(B_W(p,r))$ and we have the inequality. \qed 

\medskip 
\begin{rem}
In the same way, we obtain the following inequality from Lemma \ref{pp-e} and Corollary \ref{cor-ub} for positive even number $p$:
$$u(W(p,np+r)) \leq \frac{np(p-1)}{2}+\frac{(p-1)r}{2}-1.$$
In fact, the inequality is trivial for even number $p$ since it is equivalent to $u(W(p,np+r)) \leq c(W(p,np+r))/2 -1$, the inequality of Corollary \ref{cor-uc-1}. 
\end{rem}

\noindent Next, we show Theorem \ref{thm-42-2}. 

\medskip 
\noindent {\it Proof of Theorem \ref{thm-42-2}}. 
In the same way to the proof of Theorem \ref{thm-42-1}, the first inequality follows Lemmas \ref{pp-o} and \ref{u-r1}. 
The second one follows Lemmas \ref{pp-o} and \ref{u-r2}. 
The third one follows Lemmas \ref{pp-e} and \ref{u-r1}.   \qed

\medskip 
Murasugi~\cite{KM65} proved that the half of the absolute value of the knot signature gives a lower bound of the unknotting number.
From~\cite{MS}, we know the following formula for the signature of weaving links.
\begin{thm}[\cite{MS}]
For the weaving link $W(p,q)$, the signature is given by
$$ \sigma(W(p,q)) =
\begin{cases}
-q+1, & \text{if } p \text{ is even},\\
0, & \text{if } p \text{ is odd}.
\end{cases}$$
\label{thm-MS}
\end{thm}
\begin{rem}
By applying Theorem~\ref{thm-MS} to the weaving knots considered in Theorems~\ref{thm-42-1} and \ref{thm-42-2}, we obtain that for an even integer $p$,
$$ u(W(p,np+1)) \geq \frac{np}{2} $$
\end{rem}
\noindent and because the signature vanishes for all the other cases, it does not provide any lower bound of the unknotting number for them.
From this lower bound and Theorem~\ref{thm-42-2}, it follows that $\frac{np}{2}\leq u(W(n,np+1)) \leq\frac{np(p-1)}{2}$ if $p$ is even and $n$ is any positive integer.

\section{Region unknotting number}%%%%%%%%%%%
\label{section-rcc}

In this section, we investigate the region unknotting numbers for weaving knots and links. 
Recently, the following inequality was shown in \cite{CS}. 

\begin{thm}[\cite{CS}]
For any knot $K$, 
$$u_R(K) \leq \frac{c(K)+1}{2}.$$ 
\label{thm-CS}
\end{thm}

\noindent At the moment, no examples of a knot $K$ or a knot diagram $D$ satisfying the equality $u_R(K) = (c(K)+1)/2$ or $u_R(D) = (c(D)+1)/2$ have been found. 
The region unknotting number of some specific types of knots has been studied by the first author in \cite{ASr} and by Siwach and Prabhakar in \cite{SMt, SM2} and sharp upper bounds have been found for them. 
For instance, the region unknotting numbers of all prime knots with up to $8$ crossings are determined in~\cite{ASr}.
For upper bounds of the region unknotting numbers of torus knots and $2$-bridge knots in certain cases, one may refer to the Siwach and Prabhakar's papers ~\cite{SMt, SM2}.
In this section we study the region unknotting number of weaving knots, and prove Theorem~\ref{thm-r12} in Subsection \ref{subsection-r-k}. 
We also consider the region unlinking number of weaving links in Subsection \ref{subsection-r-l}. 
From Theorem \ref{thm-CS}, we have the following inequality for weaving knots. 

\begin{cor}
For each weaving knot $W(p,q)$, we have 
$$u_R(W(p,q)) \leq \frac{(p-1)q+1}{2}.$$
\label{cor-ur-c12}
\end{cor}

\noindent Recall that each weaving knot has a unique minimal diagram, as shown in Proposition \ref{umd}.

\subsection{Region unknotting number of weaving knots}%%%%%%%%%%%%%%
\label{subsection-r-k}

For the closure $D$ of a braid diagram $B_W(p,q)$, we call the $j$th crossing from the left-hand side in the $i$th round $c^i_j$. 
We call the $j$th region between $i$th and $(i+1)$th round $r^i_j$, and the regions outside the braid $s_1$ and $s_2$, as shown in Figure \ref{fig-w34}.
\begin{figure}[ht]
\centering
\includegraphics[width=3cm]{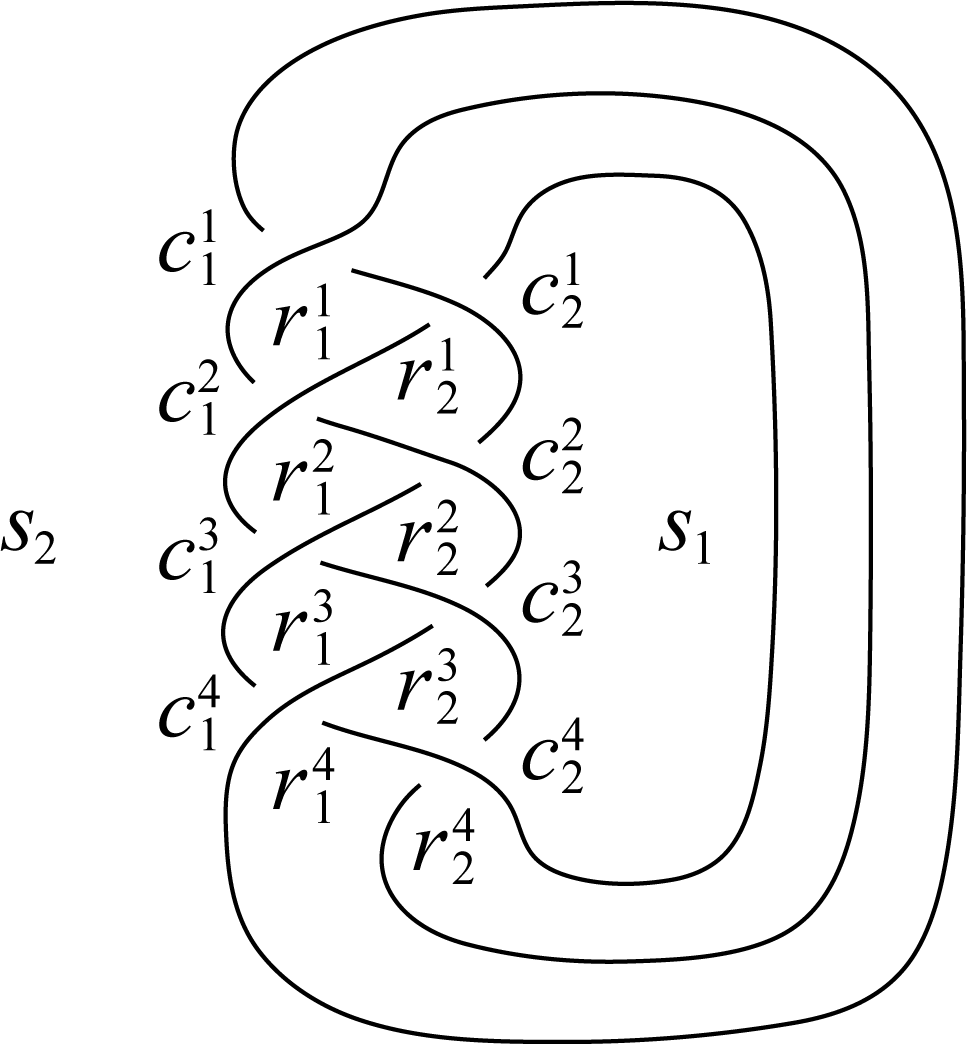}
\caption{Crossings $c^i_j$ and regions $r^i_j$, $s_i$ of $D$.}
\label{fig-w34}
\end{figure}

For a braid diagram $B$, we define the {\it region warping degree} $d_R(B)$ of $B$ to be the minimum number of region crossing changes on regions bounded by strands which transform $B$ into a braid diagram of warping degree zero. 
If $B$ cannot be transformed into a warping degree zero braid by any region crossing changes as shown in Figure \ref{fig-w23}, we assume $d_R(B)$ is not defined for $B$. 
We note that when we consider $d_R(B_W(p,q))$, we cannot use the regions $r^q_j$ ($j=1,2, \dots ,p-1$), $s_1$ and $s_2$. 
We have the following lemma. 
\begin{figure}[ht]
\centering
\includegraphics[width=0.8cm]{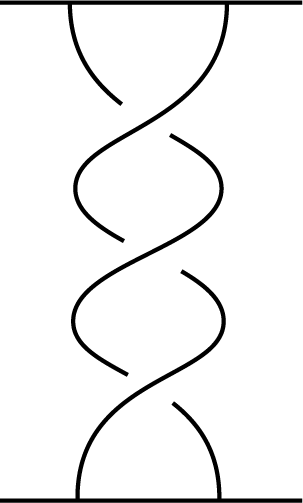}
\caption{A braid diagram $B$ such that $d_R(B)$ is undefined. }
\label{fig-w23}
\end{figure}

\begin{lem}
When $p$ is a positive odd number, 
$$d_R(B_W(p,p)) \leq \frac{p^2 -1}{4}.$$
\label{prop-drpp}
\end{lem}

\begin{proof}
Take $\mathbf{b}=(b_1, b_3, \dots , b_p, b_2, b_4, \dots , b_{p-1})$ as a sequence of base points of $B=B_W(p,p)$, which is same to the proof of Lemma \ref{pp-o}. 
By that proof, the warping crossing points of $B_{\mathbf{b}}$ are $c^i_j$ for all odd numbers $i$, $j$ satisfying $i \leq j$, $1 \leq i \leq p-2$ and $1 \leq j \leq p-i-1$, and $c^i_j$ for all even numbers $i$, $j$ satisfying $i \geq j$, $2 \leq i \leq p-1$ and $p+1-i \leq j \leq p-1$. 
When $p \equiv 1 \pmod{4}$, take a set of regions $R_1$ as follows. 
\begin{align*}
R_1= \{ r^i_j \ | \ & i \not\equiv 0 \pmod{4}, \\
& \text{when } i \equiv 1 \pmod{4}, \ j= 1, 2, 5, 6, \dots , p-i-3, p-i-2, \\
& \text{when } i \equiv 2 \pmod{4}, \ j= 2, 3, 6, 7, \dots , p-3, p-2, \\
& \text{when } i \equiv 3 \pmod{4}, \ j= p-i+1, p-i+2, \dots , p-2, p-1 \}.
\end{align*}
When $p \equiv 3 \pmod{4}$, take a set of regions $R_2$ as follows. 
\begin{align*}
R_2= \{ r^i_j \ | \ & i \not\equiv 3 \pmod{4}, \\
& \text{when } i \equiv 1 \pmod{4}, \ j= 1, 2, 5, 6, \dots , p-2, p-1, \\
& \text{when } i \equiv 2 \pmod{4}, \ j= 2, 3, 6, 7, \dots , p-i-3, p-i-2, \\
& \text{when } i \equiv 0 \pmod{4}, \ j= p-i+1, p-i+2, \dots , p-3, p-2 \}.
\end{align*}
Then by region crossing changes at the regions in $R_1$ or $R_2$, all the warping crossing points of $B_{\mathbf{b}}$ are changed and $B$ becomes warping degree zero. 
We can easily check that the number of regions in $R_1$ or $R_2$ is same to the number of the warping crossing points from the diagram, as shown in Figure \ref{fig-w99}. 
\begin{figure}[ht]
\centering
\includegraphics[width=12cm]{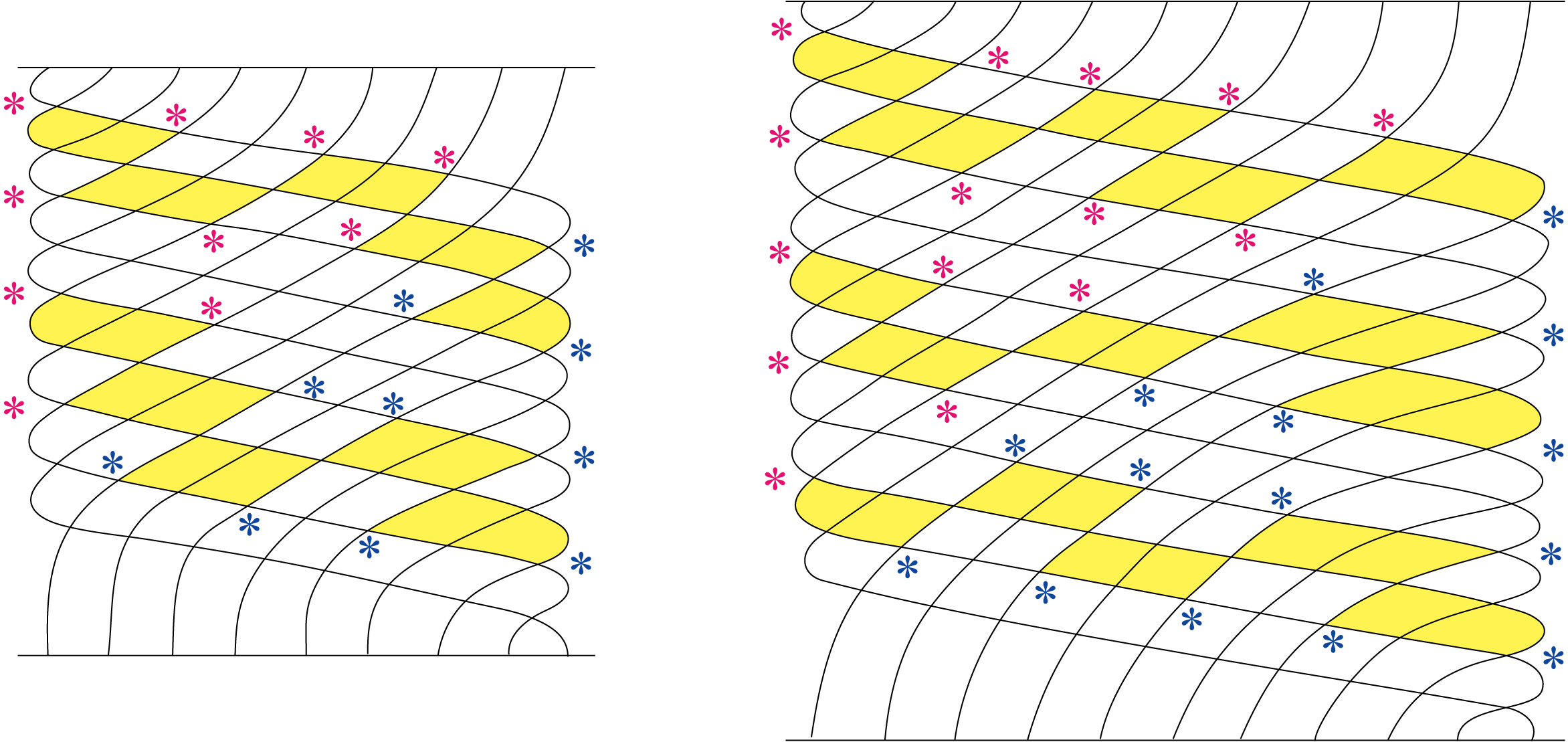}
\caption{$B_W(9,9)$ and $B_W(11,11)$. }
\label{fig-w99}
\end{figure} 
\end{proof}

\noindent We define the {\it region unlinking number} $u_R(B)$ of a braid diagram $B$ to be the minimum number of region crossing changes on $B$ which are needed to transform $B$ into a braid diagram whose closure represents a trivial knot or link. 
We note that $u_R(B)$ is undefined for some braid diagrams. 
We have the following lemma. 

\begin{lem}
When $p$ is a positive odd number, 
$$u_R(B_W(p,2)) \leq \frac{p-1}{2}.$$
\label{urb-p2}
\end{lem}

\begin{proof}
When $p \equiv 1 \pmod{4}$, apply region crossing changes at the $\frac{p-1}{2}$ regions $r^1_j$ for $j=2,3, 6,7, \dots , p-3, p-2$, as shown in Figure \ref{fig-w92}. 
Then we obtain a braid $(\sigma_1 \sigma_2 \cdots \sigma_{p-1})(\sigma_1^{-1} \sigma_2^{-1} \cdots \sigma_{p-1}^{-1})$, and similarly to the proof of Lemma \ref{u-r2}, all the crossings can be reduced by Reidemeister moves in the closure. 
When $p \equiv 3 \pmod{4}$, apply region crossing changes at the $\frac{p-1}{2}$ regions $r^1_j$ for $j=1,4,5,8,9, \dots , p-2$. 
Then we obtain a braid $(\sigma_1^{-1} \sigma_2 \cdots \sigma_{p-1})(\sigma_1^{-1} \sigma_2^{-1} \cdots \sigma_{p-1}^{-1})$, and all the crossings can be reduced by Reidemeister moves in the closure, too. 
\begin{figure}[ht]
\centering
\includegraphics[width=8cm]{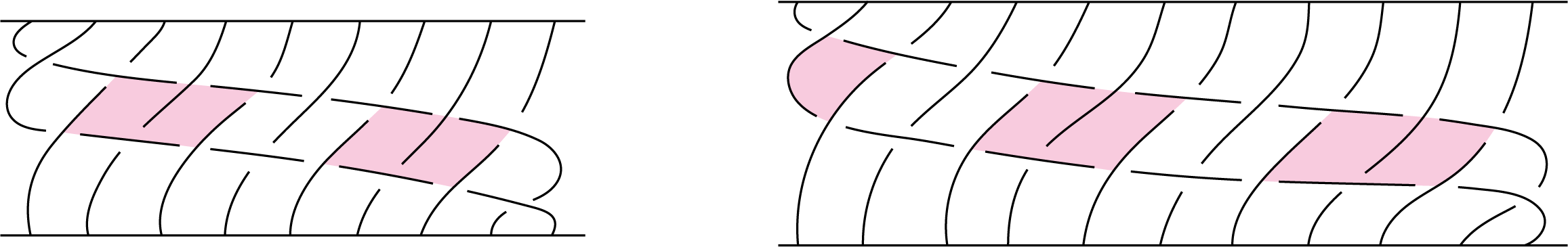}
\caption{$B_W(9,2)$ and $B_W(11,2)$. }
\label{fig-w92}
\end{figure} 
\end{proof}

\noindent Since $u_R(K) \leq u_R(D) \leq u_R(B)$ holds when $D$ is a minimal diagram of a knot $K$ and the closure of a braid diagram $B$, and $c(W(p,2))=2(p-1)$, we have the following inequality about the relation between the region unknotting number and the crossing number of the weaving knots $W(p,2)$. 

\begin{cor}
For the weaving knot $K=W(p,2)$ with any odd integer $p \geq 3$, 
$$u_R(K) \leq \frac{c(K)}{4}.$$
\end{cor}

\noindent Now we prove Theorem \ref{thm-r12}. 

\medskip 
\noindent {\it Proof of Theorem \ref{thm-r12}.} 
Each weaving knot $W(p,q)$ has a unique minimal diagram, which is the closure of $B_W(p,q)$ by Proposition \ref{umd}. 
In the same way to the proof of Theorem \ref{thm-42-1}, decompose $B_W(p, np+r)$ into $n$ braids $B_W(p,p)$ and a single $B_W(p,r)$ for $r=1,2$. 
The inequalities follow Lemmas \ref{prop-drpp}, \ref{urb-p2}, \ref{d0-t}, and \ref{u-r1}. \qed

\subsection{Region unlinking number of weaving links}
\label{subsection-r-l}

A link $L=K_1 \cup K_2 \cup \dots \cup K_r$ of $r$ components is said to be {\it proper} if 
$$\sum_{j \neq i} lk(K_i, K_j) \equiv 0 \pmod{2}$$
holds for each $i=1, 2, \dots , r$, where $lk(K_i, K_j)$ denotes the linking number of $K_i$ and $K_j$. 
It is shown by Cheng in \cite{ZC} that any diagram of a link $L$ can be transformed into a diagram of a trivial link by a finite number of region crossing changes if and only if $L$ is a proper link. 
By Proposition \ref{pp-c}, we have the following corollary. 

\begin{cor}
A weaving link $W(p,np)$ is proper if and only if $p$ is odd or $n$ is even. 
\end{cor}

\begin{proof}
Let $D= D_1 \cup D_2 \cup \dots \cup D_p$ be the closure of a braid diagram $B_W(p,np)$. 
When $p$ is an odd number, by the proof of Proposition \ref{pp-c}, $lk(D_i, D_j)=0$ holds for any pair of $i \neq j$. 
Hence $W(p,np)$ is proper when $p$ is odd. 
When $p$ is an even number, by the proof of Proposition \ref{pp-c}, $lk(D_i, D_j)= \pm n$ for each pair of $i \neq j$. 
Since $D$ represents a $p$-component link and therefore each $D_i$ interacts odd number of other components, $\sum_{j \neq i} lk(D_i, D_j) \equiv n \pmod{2}$. 
Hence, $W(p,np)$ is proper if and only if $n$ is even when $p$ is even. 
\end{proof}

\noindent Let $u_R(L)$ denote the minimal number of region crossing changes which are needed to obtain a diagram of a trivial link from a minimal diagram of a proper link $L$. 
By Lemma \ref{prop-drpp}, we have the following. 

\begin{prop}
When $p$ is an odd number, 
$$u_R(W(p, np))\leq \frac{n(p^2 -1)}{4}.$$ 
\end{prop}

\noindent In the following example, we give an algorithm for the weaving links $W(p,np)$ with even numbers $p$, $n$ to find a set of regions which unlink it by region crossing changes. 

\begin{ex}
Let $B_W(p,np)$ be a braid diagram of a weaving link of  even numbers $p$, $n$. 
Divide $B_W(p,np)$ into $\frac{n}{2}$ braids where each one is $B_W(p,2p)$. 
For a braid diagram $B_W(p,2p)$, take a sequence of base points $\mathbf{b}=(b_1, b_2, \dots , b_p)$, where $b_1, b_2, \dots , b_p$ are base points given on the top of strands from left to right. 
We denote by $s_i$ the strand with the base point $b_i$. 
By Proposition \ref{pp-c}, each pair of strands $s_i$ and $s_j$ ($i \neq j$) has exactly two warping crossing points between them in $B_W(p,2p)$ with $\mathbf{b}$. 
Let $R_{i j}$ be the set of all regions of $B_W(p,2p)$ which is bounded by the two paths of $s_i$ and $s_j$ between the two warping crossing points, as shown in Figure \ref{fig-w6}. 
Then, region crossing changes at all the regions in $R_{i j}$ change exactly the two warping crossing points. 
Take a symmetric difference for all $R_{i j}$ ($i<j$). 
We obtain a set of regions which changes all the warping crossing points of $(B_W(p,2p))_{\mathbf{b}}$. 
\begin{figure}[ht]
\centering
\includegraphics[width=8cm]{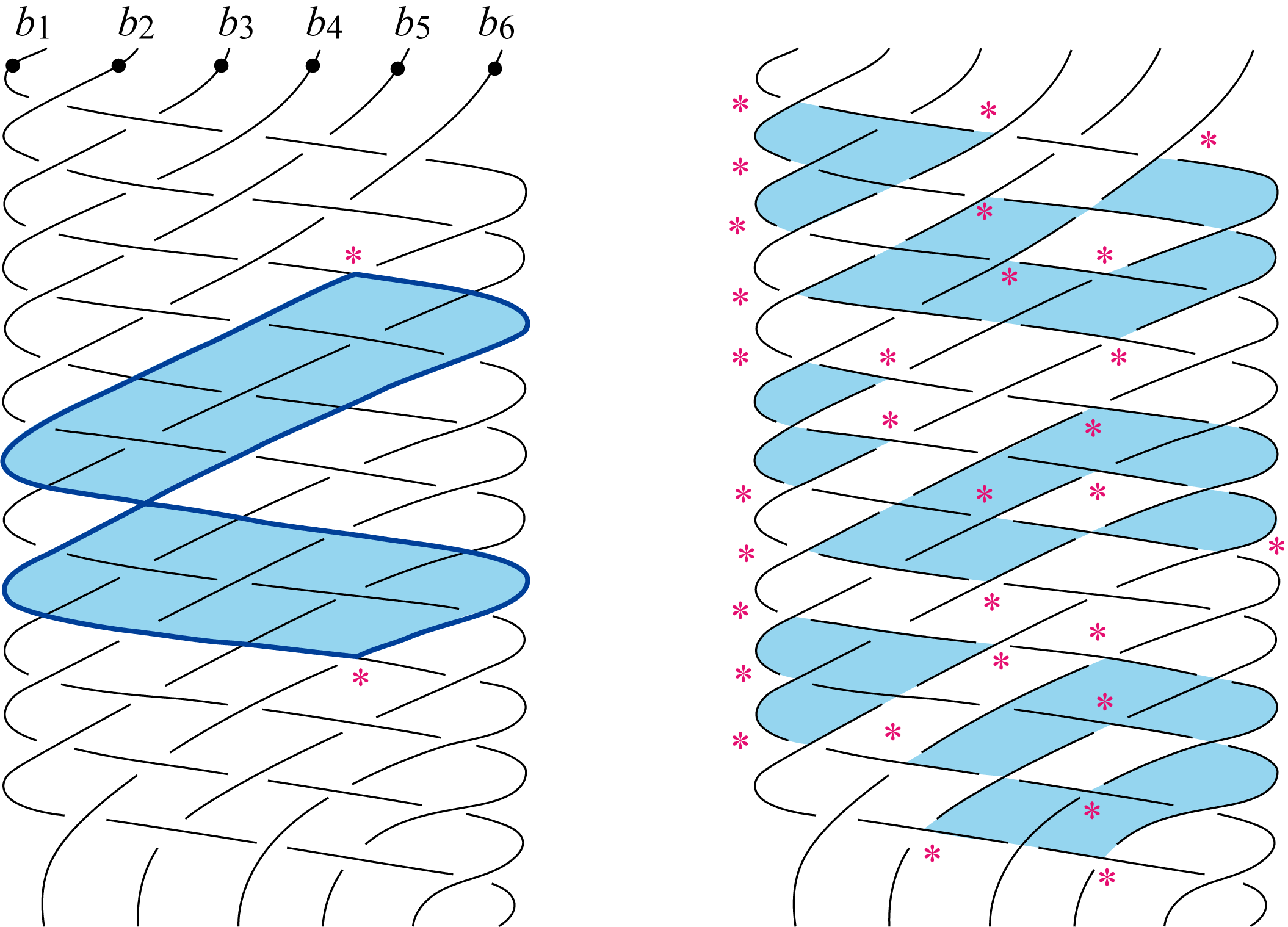}
\caption{The regions in $R_{1 3}$ are shaded on the left-hand side. A set of the regions which changes all the warping crossing points of $(B_W(6,12))_{\mathbf{b}}$ are shaded on the right-hand side. }
\label{fig-w6}
\end{figure}
\end{ex}

\appendix
\section{Isolate-region number and warping degree of weaving knot diagrams}%%%%%%%%%

In this appendix, we investigate the isolate-region number of weaving link diagrams and give a lower bounds for the warping degree. 
When $D$ is an oriented alternating knot diagram, we can directly obtain the value of $d(D)$ as $d(D_b)$ by taking a base point $b$ just before an over-crossing (\cite{ASw}). 
Moreover, when $D$ is alternating, we can also obtain the warping degree with orientation reversed as $d(-D)=c(D)-d(D)-1$ (\cite{ASw}). 
For an oriented minimal diagram $D$ of each weaving knot $W(p,q)$, we can obtain the value of $d(D)$ and $d(-D)$ by the above procedure.  

\begin{ex}
Let $d(W(p,q))$ denote the warping degree of the minimal diagram of the weaving knot $W(p,q)$ with an orientation of smaller warping degree. We have $d(W(3, 3n+1))=2n$, $d(W(3, 3n+2))=2n+1$, $d(W(4, 4n+1))=6n+1$, $d(W(4, 4n+3))=6n+4$, $d(W(5, 5n+1))=8n+1$, $d(W(5, 5n+2))=6n+2$, $d(W(5, 5n+3))=6n+3$ and $d(W(5, 5n+4))=8n+6$.
\end{ex}

\noindent The {\it isolate-region number} of a link diagram $D$, denoted by $I(D)$, is the maximum number of regions of $D$ such that any pair of regions share no crossings (\cite{TA}). 
It has been shown  that the inequality $I(D)-1 \leq d(D) \leq c(D)-I(D)$ holds for any oriented alternating knot diagram $D$ (see \cite{TA} and \cite{AA}. The relation on non-alternating diagrams are also discussed in \cite{TA}.) 
Using this inequality, for alternating knot diagrams, we can estimate the warping degree without traveling the diagram or considering the appropriate sequence of base points which follow the closure for braid diagrams. 
We have the following. 

\begin{prop}
Let $D$ be the minimal diagram of a weaving link $W(p,q)$. Then,
\begin{align*}
I(D) \geq 
\begin{cases}
\frac{(p-1)q}{4} & \text{if $q \equiv 0 \pmod{4}$,} \\
\frac{(p-1)(q-1)}{4} & \text{if $q \equiv 1 \pmod{4}$,} \\
\frac{(p-1)(q-2)}{4} + \left\lfloor \frac{p+2}{4}\right\rfloor  & \text{if $q \equiv 2 \pmod{4}$,} \\
\frac{(p-1)(q-3)}{4} + \left\lfloor \frac{p+2}{4}\right\rfloor + \left\lfloor \frac{p+1}{4}\right\rfloor & \text{if $q \equiv 3 \pmod{4}$. }
\end{cases}
\end{align*}
\label{prop-id}
\end{prop}

\begin{proof}
For $1 \leq i \leq 4 \lfloor \frac{q}{4} \rfloor$, take all the regions $r^i_j$ with $i \equiv j \pmod{4}$, as shown in Figure \ref{fig-w11}. 
In $1 \leq i \leq 4$, the number of taken regions is 
$$f(p)=  \left\lfloor \frac{p+2}{4}\right\rfloor + \left\lfloor \frac{p+1}{4}\right\rfloor + \left\lfloor \frac{p}{4}\right\rfloor + \left\lfloor \frac{p-1}{4}\right\rfloor$$
and we have $f(p)=p-1$ since $f(3)=2$ and by
\begin{align*}
f(k)=  & \left\lfloor \frac{k+2}{4}\right\rfloor + \left\lfloor \frac{k+1}{4}\right\rfloor + \left\lfloor \frac{k}{4}\right\rfloor + \left\lfloor \frac{k-1}{4}\right\rfloor , \\
f(k+1)=  & \left\lfloor \frac{k+3}{4}\right\rfloor + \left\lfloor \frac{k+2}{4}\right\rfloor + \left\lfloor \frac{k+1}{4}\right\rfloor + \left\lfloor \frac{k}{4}\right\rfloor ,
\end{align*}
we can see $f(k+1)-f(k)=1$. 
Hence in $1 \leq i \leq 4 \lfloor \frac{q}{4} \rfloor$, the number of regions $r^i_j$ with $i \equiv j \pmod{4}$ is $\lfloor \frac{q}{4} \rfloor (p-1)$. 
When $q \equiv 2, 3 \pmod{4}$, we can take 1, 2 more rounds, and the number of the regions $r^i_j$ with $i \equiv j \pmod{4}$ is $\left\lfloor \frac{p+2}{4}\right\rfloor$, $\left\lfloor \frac{p+2}{4}\right\rfloor + \left\lfloor \frac{p+1}{4}\right\rfloor$, respectively. 
Hence we have 
\begin{align*}
I(D) \geq 
\begin{cases}
\left\lfloor \frac{q}{4} \right\rfloor (p-1) & \text{if $q \equiv 0$ or $1 \pmod{4}$,} \\
\left\lfloor \frac{q}{4} \right\rfloor (p-1) + \left\lfloor \frac{p+2}{4}\right\rfloor  & \text{if $q \equiv 2 \pmod{4}$,} \\
 \left\lfloor \frac{q}{4} \right\rfloor (p-1) + \left\lfloor \frac{p+2}{4}\right\rfloor + \left\lfloor \frac{p+1}{4}\right\rfloor & \text{if $q \equiv 3 \pmod{4}$. }
\end{cases}
\end{align*}
When $q \equiv 0, 1, 2, 3 \pmod{4}$, $\lfloor \frac{q}{4} \rfloor = \frac{q}{4}, \frac{q-1}{4}, \frac{q-2}{4}, \frac{q-3}{4}$, respectively. 
\begin{figure}[ht]
\centering
\includegraphics[width=5cm]{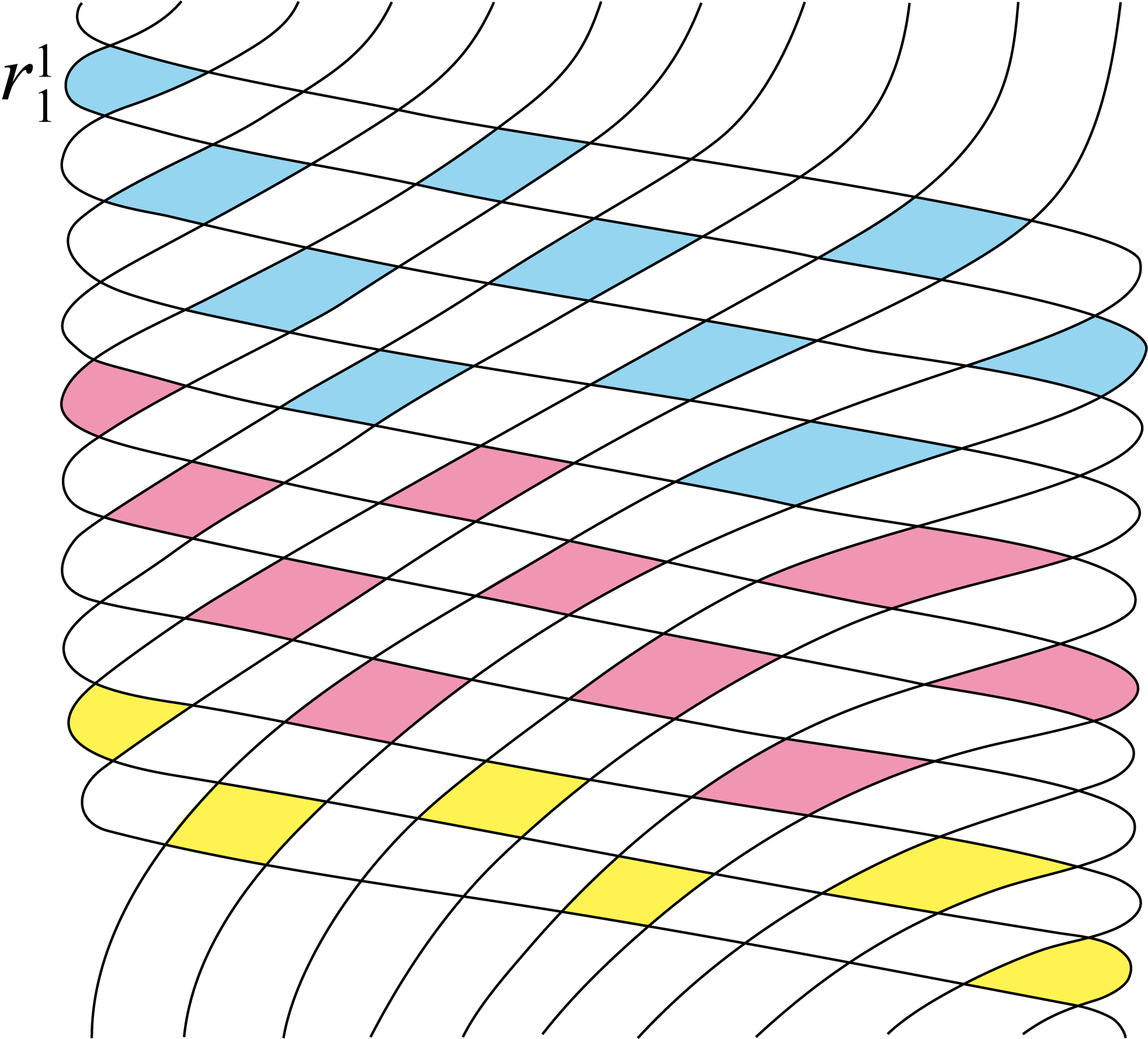}
\caption{$W(11,11)$, the case of $q \equiv 3 \pmod{4}$. }
\label{fig-w11}
\end{figure} 
\end{proof}

\noindent Since $I(D)-1 \leq d(D)$  holds for oriented alternating knot diagrams $D$, we have a lower bound for the warping degree of a weaving knot diagram from Proposition \ref{prop-id}. 

\begin{cor}
For the minimal diagram $D$ of a weaving knot $W(p,q)$, we have 
\begin{align*}
d(D) \geq 
\begin{cases}
\frac{(p-1)q}{4} -1 & \text{if $q \equiv 0 \pmod{4}$, }\\
\frac{(p-1)(q-1)}{4} -1 & \text{if $q \equiv 1 \pmod{4}$, }\\
\frac{(p-1)(q-2)}{4} + \left\lfloor \frac{p+2}{4}\right\rfloor -1 & \text{if $q \equiv 2 \pmod{4}$, }\\
\frac{(p-1)(q-3)}{4} + \left\lfloor \frac{p+2}{4}\right\rfloor + \left\lfloor \frac{p+1}{4}\right\rfloor -1 & \text{if $q \equiv 3 \pmod{4}$. }\\
\end{cases}
\end{align*}
\label{cor-wd-l}
\end{cor}

\noindent The {\it minimal warping degree} of a knot $K$, denoted by $md(K)$, is the minimal value of the warping degree $d(D)$ for all oriented minimal diagrams $D$ of $K$ (\cite{JS}). 
In \cite{ASp}, all the prime alternating knots of minimal warping degree one and two are determined. 
In the following corollary, we determine all the weaving knots of minimal warping degree up to three. 

\begin{cor}
Let $W(p,q)$ be a weaving knot with $p \geq 3$, $q \geq 2$, $\gcd (p,q)=1$. 
\begin{description}
\item[(1)] $md(W(p,q))=1$ if and only if $(p,q)= (3,2)$. 
\item[(2)] $md(W(p,q))=2$ if and only if $(p,q)=(3,4)$ or $(5,2)$. 
\item[(3)] $md(W(p,q))=3$ if and only if $(p,q)=(3,5), (5,3)$ or $(7,2)$. 
\end{description}
\end{cor}

\begin{proof}
All the pairs $(p,q)$ with $p \geq 3$, $q \geq 2$, $\gcd (p,q)=1$ such that the right-hand side of the inequality of Corollary \ref{cor-wd-l} is less than or equal to three are $(3,2)$, $(3,4)$, $(3,5)$, $(3,7)$, $(3,8)$, $(4,3)$, $(4,5)$, $(5,2)$, $(5,3)$, $(5,4)$, $(7,2)$, $(7,3)$, $(8,3)$, $(9,2)$, $(11,2)$, $(13,2)$, $(15,2)$, $(17,2)$. 
By counting the warping degrees for each minimal diagram with both orientations (we can use the equality $d(-D)= c(D)-d(D)-1$), we can find out all the weaving knots of minimal warping degree one, two and three.  
\end{proof}

\begin{rem}
From the inequality $d(D) \leq c(D)-I(D)$ for oriented alternating knot diagrams $D$, we also have an upper bound for $d(D)$ for the minimal diagram $D$ of a weaving knot $W(p,q)$ as follows: 

\begin{align*}
d(D) \leq 
\begin{cases}
\frac{3(p-1)q}{4} 
& \text{if $q \equiv 0 \pmod{4}$, }\\
\frac{(p-1)(3q+1)}{4}
& \text{if $q \equiv 1 \pmod{4}$, }\\
\frac{(p-1)(3q+2)}{4} - \left\lfloor \frac{p+2}{4}\right\rfloor 
& \text{if $q \equiv 2 \pmod{4}$, }\\
\frac{(p-1)(3q+3)}{4} - \left\lfloor \frac{p+2}{4}\right\rfloor - \left\lfloor \frac{p+1}{4}\right\rfloor
& \text{if $q \equiv 3 \pmod{4}$. }\\
\end{cases}
\end{align*}
\label{d-up}

\noindent However, we already have a better upper bound $d(D) \leq \frac{c(D)-1}{2}$ for any $D$ with $c(D) \geq 1$ (see, for example, \cite{ASw, ASu}). 
If we can find more effective upper bounds for $d(D)$ or lower bounds for $I(D)$, it would be useful to estimate the unknotting number from a point of view of warping degree. 
\end{rem}

\section*{Acknowledgment}

The authors are deeply grateful to Dr.~Madeti Prabhakar for his support and encouragement during their stay at IIT Ropar in May 2024. 
The first author's work was partially supported by JSPS KAKENHI Grant Number JP21K03263.

\end{document}